\documentclass[12pt, reqno]{amsart}

\usepackage{latexsym}
\usepackage{amssymb}
\usepackage{mathrsfs}
\usepackage{amsmath}
\usepackage{fancybox,color}
\usepackage{enumerate}
\usepackage[latin1]{inputenc}
\usepackage{eurosym}
\newcommand{\kom}[1]{}
\renewcommand{\kom}[1]{{\bf [#1]}}

 \def\1{\raisebox{2pt}{\rm{$\chi$}}}

\newtheorem{theorem}{Theorem}[section]
\newtheorem{corollary}[theorem]{Corollary}
\newtheorem{lemma}[theorem]{Lemma}

\newtheorem{remark}[theorem]{Remark}

\newcommand{\RR}{{\mathbb R}}

\newcommand{\N}{{\mathbb N}}
\newcommand{\Z}{{\mathbb Z}}

 \def\1{\raisebox{2pt}{\rm{$\chi$}}}
 
\def\vint_#1{\mathchoice%
          {\mathop{\kern 0.2em\vrule width 0.6em height 0.69678ex depth -0.58065ex
                  \kern -0.8em \intop}\nolimits_{\kern -0.4em#1}}%
          {\mathop{\kern 0.1em\vrule width 0.5em height 0.69678ex depth -0.60387ex
                  \kern -0.6em \intop}\nolimits_{#1}}%
          {\mathop{\kern 0.1em\vrule width 0.5em height 0.69678ex
              depth -0.60387ex
                  \kern -0.6em \intop}\nolimits_{#1}}%
          {\mathop{\kern 0.1em\vrule width 0.5em height 0.69678ex depth -0.60387ex
                  \kern -0.6em \intop}\nolimits_{#1}}}
\def\vintslides_#1{\mathchoice%
          {\mathop{\kern 0.1em\vrule width 0.5em height 0.697ex depth -0.581ex
                  \kern -0.6em \intop}\nolimits_{\kern -0.4em#1}}%
          {\mathop{\kern 0.1em\vrule width 0.3em height 0.697ex depth -0.604ex
                  \kern -0.4em \intop}\nolimits_{#1}}%
          {\mathop{\kern 0.1em\vrule width 0.3em height 0.697ex depth -0.604ex
                  \kern -0.4em \intop}\nolimits_{#1}}%
          {\mathop{\kern 0.1em\vrule width 0.3em height 0.697ex depth -0.604ex
                  \kern -0.4em \intop}\nolimits_{#1}}}

\newcommand{\intav}{\vint}
\newcommand{\aveint}[2]{\mathchoice%
          {\mathop{\kern 0.2em\vrule width 0.6em height 0.69678ex depth -0.58065ex
                  \kern -0.8em \intop}\nolimits_{\kern -0.45em#1}^{#2}}%
          {\mathop{\kern 0.1em\vrule width 0.5em height 0.69678ex depth -0.60387ex
                  \kern -0.6em \intop}\nolimits_{#1}^{#2}}%
          {\mathop{\kern 0.1em\vrule width 0.5em height 0.69678ex depth -0.60387ex
                  \kern -0.6em \intop}\nolimits_{#1}^{#2}}%
          {\mathop{\kern 0.1em\vrule width 0.5em height 0.69678ex depth -0.60387ex
                  \kern -0.6em \intop}\nolimits_{#1}^{#2}}}

\newcommand{\rn}{\mathbb{R}^n}

\begin{document}
\title[%The variation of 
%Endpoint Sobolev 
Continuity for maximal operators]{Endpoint Sobolev and BV Continuity for maximal operators, II
}
\author{
%Hannes Luiro and 
Jos{\'e} Madrid}

\date{\today}
\subjclass[2010]{42B25, 26A45, 46E35, 46E39.}
\keywords{Fractional maximal operator, Discrete maximal operator, Functions of bounded variation, Sobolev spaces}

%\address{Department of Mathematics and Statistics, University of Jyvaskyla,P.O.Box 35 (MaD), 40014 University of Jyvaskyla, Finland}
%\email{hannes.s.luiro@jyu.fi}
\address{Department of Mathematics, Aalto University, P.O. Box 11100, FI--00076 Aalto University, Finland}
\email{jose.madridpadilla@aalto.fi}

\address{The Abdus Salam International Centre for Theoretical Physics, Str. Costiera 11, 34151 Trieste, Italy}
\email{jmadrid@ictp.it}

\maketitle
%{\small \textbf{Abstract.} 
\begin{abstract}In this paper we study some questions about the continuity of classical and fractional maximal operators in the Sobolev space $W^{1,1}$, in both continuous and discrete setting, giving a positive answer to two questions posed recently, one of them regarding the
continuity of the map $f \mapsto \big(\widetilde M_{\beta}f\big)'$ from $W^{1,1}(\mathbb{R})$ to $L^q(\mathbb{R})$, for $q=\frac{1}{1-\beta}$. Here $\widetilde M_{\beta}$ denotes the non-centered fractional maximal operator on $\RR$ with $\beta\in(0,1)$. The second one regarding the continuity of the discrete centered maximal operator in the space of functions of bounded variation $BV(\Z)$, complementing some recent boundedness results.
\end{abstract}

\section{Introduction}

In this paper we continue with the program started in \cite{CMP}, proving two results related to the continuity of maximal operators in the continuous and discrete setting. %Moreover, 
We will use the same notation and terminology as in \cite{CMP} to facilitate the references.

\subsection{Continuous setting}
%-----------------------------------------------------

The regularity of maximal operators has been broadly study during last years for many aunthors. The starting point of this theory was the boundedness result obtained by Kinnunen in \cite{Ki} where he proved that the classical maximal operator $M$ is bounded from $W^{1,p}(\rn)$ to $W^{1,p}(\rn)$, later, it was extended by Kinninunen and Saksman to the fractional context in \cite{KiSa}.

%-------------------------------------------------

The main object of study in this paper will be the non-centered fractional maximal operator $\widetilde M_{\beta}$ which for a given function $f\in L^1_{loc}(\rn)\,$ and $0\leq \beta<n\,$, it is defined by\footnote{The supremum is taken over closed balls.}
\begin{equation}\label{eq:1}
\widetilde M_{\beta} f(x):=\sup_{B(z,r)\ni x}
\frac{r^\beta}{|B(z,r)|}\int_{B(z,r)}|f(y)|\,dy\,=:\,\sup_{B(z,r)\ni x}r^{\beta}\intav_{B(z,r)}|f(y)|\,dy\,%=\sup_{B}
\end{equation}
for every $x\in\rn\,$. We can also consider the centered version of $\widetilde M_{\beta}$, %denoted by $M^{c}_\beta$, 
which is given by taking the supremum over all the balls centered at $x$, it is denoted by $M_{\beta}$. %In the non-fractional case $\beta=0$, we also denote $M_0=M$.
In the case $\beta=0$, we recover the classical non-centered Hardy--Littlewood maximal operator $\widetilde M$.

Kinnunen and Saksman proved that 
given $0<\beta<n$, if $p>1$ and $1/q=1/p-\beta/n$ therefore $\widetilde M_{\beta}$ is bounded from $W^{1,p}(\rn)$ to $W^{1,q}(\rn)$. The continuity of this operator follows by adapting Luiro's ideas in \cite{L}. In the case $p=1$, the previous statement is not true. This endpoint case has strongly attracted the attention of many authors. The first result regarding this case was obtained by Tanaka \cite{Ta}, who proved that the map $T_{0}$ from $W^{1,1}(\mathbb{R})$ to $L^{1}(\mathbb{R})$ given by sending $f$ to $D\widetilde Mf$ is bounded (here $D\widetilde Mf$ denotes the weak derivative of $\widetilde Mf$), that result was later improved by Aldaz and Per\'ez L\'azaro in \cite{AlPe}. The continuity of the operator $T_{0}$ was recently established by Carneiro, Madrid and Pierce \cite[Theorem 1]{CMP}. In their paper many questions were posed, one of them is the following:

\noindent{\bf Question 1.} Let $0 < \beta < 1$ and $q = 1/(1-\beta)$. Is the map $f \mapsto \big({\widetilde M}_{\beta}f\big)'$ continuous from $W^{1,1}(\mathbb{R})$ to $L^q(\mathbb{R})$?
\cite[Question E]{CMP}.

This is one of the main question of this paper. Here we give a positive answer to this question, which is the content of one of our main theorems below. 

\begin{theorem}%[Main Theorem]
\label{Main Theorem}
Given $\beta\in(0,1)$, $q=\frac{1}{1-\beta}$. The map $f \mapsto \big({\widetilde M_{\beta}}f\big)'$ is continuous from $W^{1,1}(\mathbb{R})$ to $L^q(\mathbb{R})$.
\end{theorem}

The boundedness of this operator was obtained by Carneiro and Madrid in \cite[Theorem 1]{CaMa} and it will play a fundamental role in the proof of Theorem \ref{Main Theorem}. It is important to point out that if we change the space $W^{1,1}$ for the slightly weaker space $BV(\mathbb{R})$ (the space of bounded variation functions in $\mathbb{R}$) the the operator is not continuous \cite[Theorem 3]{CMP}, although it is bounded. Then, in principle it was not clear to guess whether or not it was continuous in $W^{1,1}(\mathbb{R})$.

%strongly attracted the attention of many authors in recent years. The boundedness of the classical maximal perator on the Sobolev space $W^{1,p}(\rn)$ for $p>1$ was established by Kinnunen in \cite{Ki}. The analogous result in the fractional context was established by Kinnunen and Saksman in \cite{KiSa}: for every $0<\beta<n$ we have that $M_{\beta}$ is bounded from $W^{1,p}(\rn)$ to $W^{1,q}(\rn)$ under the relation $1/q=1/p-\beta/n$ (if $p>1$).

Theorem \ref{Main Theorem} is also true in the case $\beta=0$, it was established by Carneiro, Madrid and Pierce in \cite{CMP}, they used strongly the Tanaka's lemmas about the monotonicity of the lateral maximal operators, that was fundamental in their proof. Nothing similar has been proved for the fractional maximal operator, so new tools and ideas were needed. Here we present a different argument which allow us to get the Theorem \ref{Main Theorem} in the case $\beta>0$ (fractional case) extending the main theorem in \cite{CMP}, however this proof does not work in the case $\beta=0$. For  some related results we refer to \cite{BCHP}, \cite{CFS}, \cite{CaHu}, %\cite{CMP}, 
\cite{CaMo}, \cite{CaSv}, \cite{HM}, \cite{HO}, \cite{L}, \cite{L2} \cite{LM}, \cite{Ma}, \cite{R} and \cite{S}.

\subsection{Discrete setting} Given a function $f:\mathbb{Z} \to \mathbb{R}$. We will keep the usual notations
\begin{equation*}
\|f\|_{\ell^{p}{( \Z)}}:= \left(\sum_{n\in \Z} {|f(n)|^{p}}\right)^{1/p},
\end{equation*}
denotes its $\ell^p(\mathbb{Z})-$norm, for every $1\leq p<\infty$, and
\begin{equation*}
\|f\|_{\ell^{\infty}{(\Z)}}:= \sup_{n\in\Z}{|f(n)|}.
\end{equation*}
We define the derivative of $f$ at the point $n\in\Z$ by
$$
f'(n)=f(n+1)-f(n).
$$
We say that $f$ is a function of bounded variation if %its discrete derivative by $f'(n) := f(n+1) - f(n)$ and its total variation by
\begin{equation*}
Var(f) := \|f'\|_{l^1(\mathbb{Z})} = \sum_{n\in \Z} |f(n+1) - f(n)|<\infty.
\end{equation*}
We will denote by $BV(\Z)$ the space of functions of bounded variation, which is a Banach space with the norm 
\begin{equation}\label{BV norm discrete}
\|f\|_{{\rm BV(\Z)}} = \left|f(-\infty)\right| + Var(f),
\end{equation}
where $f(-\infty):= \lim_{n\to -\infty} f(n)$. 

\smallskip
The %main object to study in this section is the 
discrete centered Hardy-Littlewood maximal operator $M$ is defined by
%For $f:\Z \to \R$ we define the discrete uncentered Hardy-Littlewood maximal function ${\M}f:\Z \to \R^+$ by 
\begin{equation}\label{disc_HLM}
{M}f(n) = \sup_{\stackrel{r \geq 0}{r \in \Z}} \frac{1}{(2r  +1)} \sum_{k = -r}^{r} |f(n + k)|.
\end{equation}
While %Another interesting object is 
the discrete uncentered Hardy-Littlewood maximal operator $\widetilde M$ is defined by
\begin{equation}\label{disc_HLM2}
{\widetilde M}f(n) = \sup_{\stackrel{r,s \geq 0}{r,s \in \Z}} \frac{1}{(r +s +1)} \sum_{k = -r}^{s} |f(n + k)|.
\end{equation}

%The discrete analogue of the sharp inequality \eqref{Intro_1_var_UHL} was established by Bober, Carneiro, Hughes and Pierce in \cite{BCHP}, who showed that
%\begin{equation}\label{Intro_BCHP}
%Var({\M}f) \leq Var(f).
%\end{equation}
%In the centered case, inequality \eqref{Intro_BCHP} with a constant $C >1$ was obtained by Temur in \cite{Te}.

\smallskip

%Inequality \eqref{Intro_BCHP}, together with the simple fact that $|{\M}f(-\infty)| \leq |f(-\infty)|+ Var(f)$ (which follows for example from Lemma \ref{lemma_min_max} in Section \ref{sec_prelim_discrete}), establishes the boundedness of ${M} :BV(\mathbb{Z}) \to BV(\mathbb{Z})$. Our second result establishes the continuity of this map, answering the discrete analogue of Question B.
 
 The boudedness of $\widetilde M: BV(\Z)\to BV(\Z)$ was established in \cite{BCHP}, and the continuity of this operator was recently established in \cite{CMP}. Although   boundedness of $M:BV(\Z)\to BV(\Z)$ was obtained by Temur in \cite{Te}, %adapting the Kurka'ideas in \cite{Ku}
  the continuity of this operator was an open problem %until today 
 \cite[Question D]{CMP} and it is the main result of this section ($M$ is less regular than $\widetilde M$ and usually is more complicated to treat).

\begin{theorem}%[Main Theorem]
\label{Thm2}
The map ${M} :BV(\Z) \to BV(\Z)$ is continuous.
\end{theorem}

%%%%%%%%%%%%%%%%%%%%%%%%%%%%%%%%%%%%%%%%%%%%%%%%%%%5%%%
%\newpage
Taking in considerations the results obtained in this paper, the situation of the {{endpoint continuity program for maximal operators}} is the following.

%Table 1 below summarizes the results of this paper and the open problems. The word YES in a box means that we have established the continuity of the corresponding map, whereas the word NO means that we have shown it fails. The remaining boxes are marked as OPEN problems.

\begin{table}[h]
\renewcommand{\arraystretch}{1.3}
\centering
\caption{Endpoint continuity program}
\label{Table-ECP}
\begin{tabular}{|c|c|c|c|c|}
\hline
 \raisebox{-1.3\height}[0.5cm]{{------------}}&  \parbox[t]{2.4cm}{  $W^{1,1}-$\tiny{continuity}; \\ continuous setting} &  \parbox[t]{2.4cm}{  $BV-$\tiny{continuity}; \\ continuous setting} & \parbox[t]{2.4cm}{  $W^{1,1}-$\tiny{continuity}; \\ discrete setting} &  \parbox[t]{2.2cm}{  $BV-$\tiny{continuity}; \\ discrete setting} \\ [0.5cm]
\hline
 \parbox[t]{3.3cm}{ {\tiny{Centered classical}} \\ \tiny{maximal operator}} &  \raisebox{-0.8\height}{OPEN} & \raisebox{-0.8\height}{OPEN} & \raisebox{-0.8\height}{YES$^2$}  &\raisebox{-0.8\height}{{YES: Theorem \ref{Thm2}}}\\[0.5cm]
 \hline
  \parbox[t]{3.3cm}{ \tiny{Uncentered classical} \\ \tiny{maximal operator}} &  \raisebox{-0.8\height}{YES$^4$} &  \raisebox{-0.8\height}{OPEN} &  \raisebox{-0.8\height}{YES$^2$} &  \raisebox{-0.8\height}{YES$^4$} \\[0.5cm]
 \hline
 \parbox[t]{3.3cm}{ \tiny{Centered fractional} \\ \tiny{maximal operator}} &  \raisebox{-0.8\height}{OPEN$^1$} &  \raisebox{-0.8\height}{NO$^{1,}$$^4$} &  \raisebox{-0.8\height}{YES$^3$} & \raisebox{-0.8\height}{NO$^{1,}$$^4$}  \\[0.5cm]
 \hline
\parbox[t]{3.3cm}{ \tiny{Uncentered fractional} \\ \tiny{maximal operator}} &  \raisebox{-0.8\height}{YES: Theorem \ref{Main Theorem}} &  \raisebox{-0.8\height}{NO $^4$} &  \raisebox{-0.8\height}{YES$^3$}  & \raisebox{-0.8\height}{NO$^4$} \\[0.5cm]
 \hline
\end{tabular}
\vspace{0.05cm}
\flushleft{
\ \ \footnotesize{$^1$ Corresponding boundedness result not yet known.}\\
\ \ $^2$ Result proved in \cite{CaHu}.\\
\ \ $^3$ Result proved in \cite{CaMa}.\\
\ \ $^4$ Result proved in \cite{CMP}}.
\end{table}

In the table above the word YES means that the continuity was estabished, the word NO means that there are counterexamples to the continuity, and the word OPEN means that the problem is still unsolved.

%%%%%%%%%%%%%%%%%%%%%%%%%%%%%%%%%%%%%%%%%%%%%%%%%%%%%%%%%%%

\section{Continuous setting -- Preliminaries}\label{sec2}

%-----------------------------------------------------------

Through out this paper we will use the following notations. Given a function $f\in W^{1,1}(\mathbb{R})$, $f'$ denotes its weak derivative. For every $p\in[1,\infty]$ we denote by $\|f\|_{p}$ the usual norm in $L^{p}(\mathbb{R})$.\\

%\noindent
Fix $\beta\in(0,1)$. If $f$ is a function in $W^{1,1}(\mathbb{R})$, given a point $x\in\mathbb{R}$, we say that an interval $B$ is a {\it{good ball for $x$ with respect to $f$}} if $x\in B$ and
$$
\widetilde M_{\beta}f(x)=\frac{r^{\beta}}{2r}\int_{B}|f|=r^{\beta}\intav_{B}|f|.
$$
Here $r$ denotes the radius of $B$. The condition $f\in L^{1}(\mathbb{R})$ implies that for every $x\in\mathbb{R}$ there is at least one good ball.

%\noindent
Given a function $f$ in $W^{1,1}(\mathbb{R})$ and a sequence of functions $\{f_{j}\}_{j\in\mathbb{N}}\subset W^{1,1}(\mathbb{R})$, for every $x\in \mathbb{R}$ we denote by $B_{x}$ and $B_{x,j}$ a family of good balls for $x$ with respect to $f$ and $f_{j}$ respectively, and we denote by $r_{x}$ and $r_{x,j}$ the radius of these balls, we call any of these {\it{good radii}} of $f$ (or $f_j$ respectively) at the point $x$.
%\noindent
We denote by $\chi_{S}$ the characteristic function of a set $S\subset\mathbb{R}$.

%%%%%%%%%%%%%%%%%%%%%%%%%%%%%%%%%%%%%%%%%%%%5
%%%%%%%%%%%%%%%%%%%%%%%%%%%%%%%%%%%%%%%%%%%%%%%%%

The following is a basic fact, it will be useful along this paper, it is the content of Lemma 14 in \cite{CMP}.
\begin{lemma}\label{pass_modulus}
Let $f\in W^{1,1}(\RR)$ and $\{f_{j}\}_{j \geq 1}\subset W^{1,1}(\RR)$ be such that $\|f_{j}-f\|_{W^{1,1}(\RR)}\to0$ as $j\to\infty$. Then $\||f_{j}|-|f|\|_{W^{1,1}(\RR)}\to0$ as $j\to\infty$.
\end{lemma}
%\begin{proof}
%This is the content of Lemma 14 in \cite{CMP}.
%Since $\big||f_{j}|-|f|\big| \leq |f_j -f|$ pointwise, it follows that $\||f_{j}|-|f|\|_{L^{1}(\RR)}\to0$ as $j \to \infty$. Noting that $|f| = \max\{f, -f\}$, the fact that $\||f_{j}|'-|f|'\|_{L^{1}(\RR)}\to0$ follows directly from Lemma \ref{reduction to lateral max}.
%\end{proof}

%%%%%%%%%%%%%%%%%%%%%%%%%%%%%%%%%%%%%%%%%%%%%%%%%%%5

%The next Lemma establish 
Another very useful observation is the following.
\begin{lemma}\label{unif conv lemma}
Given a function $f\in W^{1,1}(\mathbb{R})$ and a sequence $\{f_{j}\}_{j\in\mathbb{N}}\subset W^{1,1}(\mathbb{R})$ such that $\|f_{j}-f\|_{W^{1,1}(\mathbb{R})}\to0$ as $j\to\infty$. For every $\beta\in [0,1)$ we have that 
$$
\|\widetilde M_{\beta}f_{j}-\widetilde M_{\beta}f\|_{\infty}\to 0\ \ \text{as} \ \ j\to \infty.
$$
\end{lemma}
\begin{proof}[Proof of Lemma \ref{unif conv lemma}]
By H\"older's Inequality we have 
\begin{eqnarray*}
&&\|\widetilde M_{\beta}f_{j}-\widetilde M_{\beta}f\|_{\infty}\\
&&\leq C\|f_{j}-f\|_{q'}
\leq C\|f_{j}-f\|^{1/q}_{\infty}\|f_{j}-f\|_{1}^{1/q'}
\leq C\|f_{j}-f\|_{W^{1,1}(\mathbb{R})}.
\end{eqnarray*}
Here $q'=\frac{q}{q-1}=\frac{1}{\beta}$ and $C>0$ is a universal constant.
\end{proof}

%%%%%%%%%%%%%%%%%%%%%%%%%%%%%%%%%%%%%%%%%%%%%%%%%%%%%
Before proceeding we recall the notion of approximately differentiable function.\\
\noindent
A function $f: \mathbb{R} \to \mathbb{R}$ is said to be {\it approximately differentiable} at a point $x_0\in \mathbb{R}$ if there exists a real number $\alpha$ such that, for any $\varepsilon >0$, the set
$$A_{\varepsilon} = \left\{ x \in \mathbb{R}; \ \frac{|f(x) - f(x_0) - \alpha(x- x_0)|}{|x-x_0|} < \varepsilon \right\}$$
has $x_0$ as a density point. In this case, the number $\alpha$ is called the approximate derivative of $f$ at $x_0$ and it is uniquely determined. It follow directly from the definitions that if $f$ is differentiable at $x_0$ then it is approximately differentiable at $x_0$, and the classical and approximate derivatives coincide.

\begin{lemma}\label{representacion de la derivada}
Given a function $f\in W^{1,1}(\mathbb{R})$ for almost every $x\in\mathbb{R}$ we have that
\begin{equation}\label{rep de derivada}
(\widetilde M_{\beta}f)'(x)=r^{\beta}_{x}\intav_{B_{x}}|f|'(y)dy.
\end{equation}
\end{lemma}
\begin{proof}[Proof of Lemma \ref{representacion de la derivada}]
Following the argument of Haj\l asz and Mal\'{y} \cite[Theorems 1 and 2]{HM} we see that if a function $f \in W^{1,1}(\mathbb{R})$ thus $\widetilde M_{\beta}f$ is approximately differentiable a.e. and the approximate derivative is equal to the right hand side of \eqref{rep de derivada} for almost every $x \in \mathbb{R}$. (and in this case for every good ball $B_x$ with good radius $r_x$). 
We can conlude using the fact that $\widetilde M_{\beta}f$ is absolutely continuous and therefore it is differentiable almost everywhere in the classical sense \cite[Theorem 1]{CaMa}.
\end{proof}

%%%%%%%%%%%%%%%%%%%%%%%%%%%%%%%%%%%%%%%%%%%%%%%%%%5
%\newb{Lema de representacion de la derivada debe aparecer aqui}.

\begin{lemma}\label{convergencia puntual donde discolan}
Let $f\in W^{1,1}(\mathbb{R})$ and $\{f_{j}\}_{j \in\mathbb{N}}\subset W^{1,1}(\mathbb{R})$ be such that $\|f_{j}-f\|_{W^{1,1}(\mathbb{R})}\to0$ as $j\to\infty$. Then
$$
(\widetilde M_{\beta}f_{j})'(x)\to (\widetilde M_{\beta}f)'(x) %\ \ \text{pointwise}
$$
for almost every $x\in \mathbb{R}$.
\end{lemma}
\begin{proof}[Proof of Lemma \ref{convergencia puntual donde discolan}]
%\newb{To Do}.
Using Lemma \ref{representacion de la derivada} and an argument of Carneiro, Madrid and Pierce shows the result \cite[Lemma 15]{CMP}.
\end{proof}

{\remark \label{remark1}
As a consequence of Lemma \ref{convergencia puntual donde discolan} and the Brezis-Lieb Lemma \cite{BL}, in order to obtain our main theorem it is sufficient to prove that
$$
\int_{\mathbb{R}}|(\widetilde M_{\beta}f_{j})'|^{q}\to \int_{\mathbb{R}}|(\widetilde M_{\beta}f)'|^{q} \ \ \text{as}\ \ j\to\infty.
$$}
%\end{remark}

%%%%%%%%%%%%%%%%%%%%%%%%%%%%%%%%%%%%%%%%%%%%%%%%%%%%%%%%%%%
{We start by analyzing the situation inside of a given compact.}
\begin{lemma}\label{compact lemma}
Given a function $f\in W^{1,1}(\mathbb{R})$ and a sequence $\{f_{j}\}_{j\in\mathbb{N}}\subset W^{1,1}(\mathbb{R})$ such that $\|f_{j}-f\|_{W^{1,1}(\mathbb{R})}\to0$ as $j\to\infty$. Given a compact $K\subset \mathbb{R}$ we have that 
$$
\int_{K}|(\widetilde M_{\beta}f_{j})'|^{q}\to\int_{K}|(\widetilde M_{\beta}f)'|^{q} \ \ \text{as} \ \ j\to\infty.
$$ 
\end{lemma}

\begin{proof}[Proof of Lemma \ref{compact lemma}]
First of all, we assume without lost of generality that $f\neq0$ (because in that case the result follows directly from the boundedness theorem \cite[Theorem 1]{CaMa}). Since $\widetilde M_{\beta}f$ and $Mf$ are continuous (see \cite[Theorem 2.5]{AlPe} and \cite[Theorem 1]{CaMa}), there are positive constants $C_{K}$ and $\overline C_{K}$ such that
$$
\inf_{x\in K}\widetilde M_{\beta}f(x)=\min_{x\in K}\widetilde M_{\beta}f(x)=C_{K}>0
$$
and
$$
\sup_{x\in K}\widetilde Mf(x)=\max_{x\in K}\widetilde Mf(x)=\overline C_{K}>0.
$$
Then by Lemma \ref{unif conv lemma} we have that there is $j_{1}(K)$ such that $\|\widetilde M_{\beta}f_{j}-\widetilde M_{\beta}f\|_{\infty}\leq C_{K}/2$ and $\|\widetilde Mf_{j}-\widetilde Mf\|_{\infty}\leq \overline C_{K}/2$ for every $j\geq j_{1}(K)$. Therefore
$$\widetilde M_{\beta}f_{j}(x)\geq \widetilde M_{\beta}f(x)-C_{K}/2\geq C_{K}/2$$ for every $j\geq j_{1}(K)$. Thus
\begin{eqnarray*}
\frac{3}{2}\overline C_{K}r^{\beta}_{j,x}\geq r^{\beta}_{j,x}\widetilde Mf_{j}(x)\geq\frac{r^{\beta}_{j,x}}{2r_{j,x}}\int_{B_{j,x}}|f_{j}|=\widetilde M_{\beta}f_{j}(x)\geq C_{K}/2.
\end{eqnarray*}
Therefore
\begin{equation}\label{low bound for radi}
r_{j,x}\geq \left(\frac{C_{K}}{3\overline C_{K}}\right)^{\frac{1}{\beta}}=:\widetilde C_{K} \ \ \text{for every} \ \ j\geq j_{1}(K), x\in K.
\end{equation}

Using \eqref{low bound for radi} we have that for every $x\in K$
\begin{eqnarray*}
|(\widetilde M_{\beta}f_{j})'(x)|&=&\frac{r^{\beta}_{j,x}}{2r_{j,x}}\left|\int_{B_{x,j}}|f_{j}|'\right|\\
&\leq&\frac{1}{2\widetilde C_{K}^{1-\beta}}\left[\int_{\mathbb{R}}||f_{j}|'-|f|'|+\int_{\mathbb{R}}||f|'|\right]\\
&\leq& \frac{1}{\widetilde C_{K}^{1-\beta}}\int_{\mathbb{R}}||f|'|, \ \ %\text{for every}\ \ j\geq\max\{j_{1}(K),j_{0}\}, x\in K.
\end{eqnarray*}
for every\ \ $j\geq\max\{j_{1}(K),j_{0}\}$. Here $j_0$ is a positive integer such that $\||f_j|'-|f|'\|_{L^{1}(\RR)}\leq \||f|'\|_{L^{1}(\RR)}$ for every $j\geq j_0$.

Therefore by the Dominated Convergence Theorem, using Lemma \ref{convergencia puntual donde discolan} we conclude that
\begin{equation}
\int_{K}|(\widetilde M_{\beta}f_{j})'|^{q}\to\int_{K}|(\widetilde M_{\beta}f)'|^{q} \ \ \text{as} \ \ j\to\infty.
\end{equation}

\end{proof}

%%%%%%%%%%%%%%%%%%%%%%%%%%%%%%%%%%%%%%%%%%%%%%%%%%%%%%%%%%%%%%%

Heuristically speaking, as a consequence of Lemma \ref{compact lemma}, in order to get our desired result it is enough to prove that given a interval $[a,b]$ with $\int_{[a,b]^c}|(\widetilde M_{\beta}f)'|^q$ ``small", we must have that $\int_{[a,b]^c}|(\widetilde M_{\beta}f_j)'|^q$ is also ``small" for every $j$ sufficiently large. %The next two lemmas allow us to prove this.

\section{Proof of  Theorem \ref{Main Theorem}
}

\begin{proof}[Proof of Theorem \ref{Main Theorem}]
Given $\varepsilon>0$, there is a real number $y>\frac{1}{\varepsilon}>0$ such that
\begin{eqnarray*}
&&\int_{-\infty}^{-y}|(f)'|<\varepsilon,\int_{y}^{\infty}|(f)'|<\varepsilon,\int_{-\infty}^{-y}|(\widetilde M_{\beta}f)'|^{q}<\varepsilon, \int_{y}^{\infty}|(\widetilde M_{\beta}f)'|^{q}<\varepsilon,\\
&& \widetilde M_{\beta}f(-4y)<\varepsilon \ \ \text{and}\ \ \ \widetilde M_{\beta}f(4y)<\varepsilon.
\end{eqnarray*}

For every $\varepsilon>0$, there is $j_{\varepsilon}$ such that 
$$
\int_{\mathbb{R}}|f'-f'_{j}|\leq \varepsilon\ \ \text{and}\ \ \|\widetilde M_{\beta}f_{j}-\widetilde M_{\beta}f\|_{L^{\infty}(\mathbb{R})}<\varepsilon \ \ \text{for every} \ \ j\geq j_{\varepsilon}.
$$ 
From now on, for every $x\in{[3y,\infty)}$, we will fix a good ball $B_x$ (and $B_{x,j}$) for $f$ (and $f_j$, respectively) at the point $x$. Given $x\geq 3y$ we define
$
B_{x}=(a_x,b_x)
$
to be a ball such that $(a_x,b_x)$ is a good ball for $f$ in $x$, $a_x$ is maximum possible and once we have found $a_x$, $b_x$ is the minimum possible. The fact that defintion is well posed follows using that we are considering functions in $L^{1}(\RR)$, details are left to interested lector. Analogously we define $a_{x,j}$ and $b_{x,j}$ for every $j$.

Therefore, given $j\geq j_{\varepsilon}$, we have
\begin{eqnarray}\label{I+II}
\int_{3y}^{\infty}|(\widetilde M_{\beta}f_{j})'(x)|^{q}dx&=&\int_{3y}^{\infty}|(\widetilde M_{\beta}f_{j})'(x)|^{q}\chi_{B^{c}_{x,j}}(y)dx\nonumber\\
&&+\int_{3y}^{\infty}|(\widetilde M_{\beta}f_{j})'(x)|^{q}\chi_{B^{}_{x,j}}(y)dx\nonumber\\
&=&I+II.
\end{eqnarray}

To estimate I, we use the boundedness result \cite[Theorem 1]{CaMa}
\begin{eqnarray}\label{I}
&&\int_{3y}^{\infty}|(\widetilde M_{\beta}f_{j})'(x)|^{q}\chi_{B^{c}_{x,j}}(y)dx\nonumber\\
&&\leq \int_{3y}^{\infty}|(\widetilde M_{\beta}(f_{j}\chi_{[y,\infty]})'(x))|^{q}\nonumber\\
&&\leq C\|f'_{j}\chi_{[y,\infty]}\|^{q}_{L^{1}(\mathbb{R})}\nonumber\\
&&\leq (\|f'_{j}-f'\|_{L^{1}(\mathbb{R})}+\|f'\chi_{[y,\infty]}\|_{L^{1}(\mathbb{R})})^{q}\nonumber\\
&&\leq(2\varepsilon)^{q}.
\end{eqnarray}

\newpage
The key ingredients to estimate II are the following
\begin{itemize}
\item Claim 1: $V_{y,j}=\{x\in(3y,\infty), y\in B_{x,j}\}$ is an open set. To see this we take a point $x\in V_{y,j}$, it implies that $a_x<y$, assuming (by contradiction) that there is not a neighboorhod of $x$ contained in $V_{y,j}$ we would have a sequence $\{x_{i}\}_{i\in\N}\subset V_{y,j}^{c}$ such that $x_{i}\to x$ as $i\to\infty$, $x_{i}\in V^{c}_{y,j}$ implies $y\leq a_{x_i},$ using the fact that $\{a_{x_i}\}_{i\in\N}$ and $\{b_{x_i}\}_{i\in\N}$ are bounded sequences, by passing to subsequences if necessary we have that $a_{x_{i}}\to a$ and $b_{x_{i}}\to b$ for some numbers $a$ and $b$ such that $(a,b)$ is a good ball for $f$ at $x$, thus $y\leq \lim_{i\to\infty} a_{x_i}=a\leq a_x$, by construction, and it is a contradiction.
\item Claim 2: For every $z$, $w\in V_{y,j}$ such that $z<w$ we have that $a_w<y<z<w\leq b_w$ therefore
$$
\widetilde M_{\beta}f(w)\leq \widetilde M_{\beta}f(z).
$$
\item Claim 3: For every $x\in V_{y,j}$ we have that 
$$r_{x,j}= \frac{b_{x,j}-a_{x,j}}{2}\geq \frac{x-y}{2}\geq y.$$
\end{itemize}
Using Claim 1 we see that $V_{{y,j}}=\{x\in(3y,\infty),y\in B_{x,j}\}=\cup_{i} (a_{i},b_{i})$, with $a_{1}<b_{1}<a_{2}<b_{2}<\dots$, then using Claim 2 we have that $\widetilde M_{\beta}f_j$ is a non-increasing function in $V_{y,j}$, in particular $(\widetilde M_{\beta}f_{j})'(x)\leq0$ for almost every $x\in V_{y,j}$ and $\widetilde M_{\beta}f_{j}(b_i)>\widetilde M_{\beta}f_{j}(a_{i+1})$, for every $i$, therefore (using Claim 3 in the sixth line of the computation below).

\begin{eqnarray}\label{II}
II&=&\int_{4y}^{\infty}|(\widetilde M_{\beta}f_{j})'(x)|^{q}\chi_{B^{}_{x,j}}(y)dx\nonumber\\
&=&\int_{W_{y}}|(\widetilde M_{\beta}f_{j})'(x)|^{q}dx\nonumber\\
&=&\sum_{i}\int_{a_{i}}^{b_{i}}|(\widetilde M_{\beta}f_{j})'(x)|^{q}dx\nonumber\\
&=&\sum_{i}\int_{a_{i}}^{b_{i}}(-(\widetilde M_{\beta}f_{j})'(x))^{q}dx\nonumber\\
&\leq& \|(\widetilde M_{\beta}f_{j})'\|^{q-1}_{L^{\infty}(W_{y})}\sum_{i}\int_{a_{i}}^{b_{i}}(-(\widetilde M_{\beta}f_{j})'(x))\nonumber\\
&\leq&\frac{C}{|y|^{\beta}}\||f_j|'\|^{q-1}_{L^{1}(\mathbb{R})}\left(\sum_{i}(\widetilde M_{\beta}f_{j}(a_{i})-\widetilde M_{\beta}f_{j}(b_i))\right)\nonumber\\
&\leq&\frac{C}{|y|^{\beta}}\||f_j|'\|^{q-1}_{L^{1}(\mathbb{R})}\widetilde M_{\beta}f_{j}(4y)\nonumber\\
&\leq&\frac{C}{|y|^{\beta}}\left(\||f_j|'-|f|'\|_{L^{1}(\mathbb{R})}+\||f|'\|_{L^{1}(\mathbb{R})}\right)^{q-1}\nonumber\\
&&\times\left(\widetilde M_{\beta}f_j(4y)-\widetilde M_{\beta}f(4y)+\widetilde M_{\beta}f(4y)\right)\nonumber\\
&\leq&C\varepsilon^{\beta}(\varepsilon+\||f|'\|_{L^{1}(\mathbb{R})})^{q-1}(2\varepsilon).
\end{eqnarray}

From \eqref{I+II}, \eqref{I} and \eqref{II} we get
\begin{eqnarray*}
\int_{4y}^{\infty}|(\widetilde M_{\beta}f_{j})'(x)|^{q}dx\leq (2\varepsilon)^{q}+C\varepsilon^{\beta}(\varepsilon+\||f|'\|_{L^{1}(\mathbb{R})})^{q-1}(2\varepsilon),
\end{eqnarray*}
for every $j\geq j_\varepsilon$. Analogously,
\begin{eqnarray*}
\int_{-\infty}^{-4y}|(\widetilde M_{\beta}f_{j})'(x)|^{q}dx\leq (2\varepsilon)^{q}+C\varepsilon^{\beta}(\varepsilon+\||f|'\|_{L^{1}(\mathbb{R})})^{q-1}(2\varepsilon).
\end{eqnarray*}

Therefore, by Lemma \ref{compact lemma}, there is $\widetilde j_{\epsilon}\geq j_\varepsilon$ such that
\begin{eqnarray*}
\int_{\mathbb{R}}|(\widetilde M_{\beta}f_{j})'(x)|^{q}dx&\leq&\int_{-3y}^{3y}|(\widetilde M_{\beta}f_{j})'(x)|^{q}dx\\
&&+2((2\varepsilon)^{q}+C\varepsilon^{\beta}(\varepsilon+\||f|'\|_{L^{1}(\mathbb{R})})^{q-1}(2\varepsilon))\\
&\leq&\int_{-3y}^{3y}|(\widetilde M_{\beta}f_{})'(x)|^{q}dx+\varepsilon\\
&&+2((2\varepsilon)^{q}+C\varepsilon^{\beta}(\varepsilon+\||f|'\|_{L^{1}(\mathbb{R})})^{q-1}(2\varepsilon))\\
&\leq&\int_{\mathbb{R}}|(\widetilde M_{\beta}f_{})'(x)|^{q}dx+\varepsilon\\
&&+2((2\varepsilon)^{q}+C\varepsilon^{\beta}(\varepsilon+\||f|'\|_{L^{1}(\mathbb{R})})^{q-1}(2\varepsilon))
\end{eqnarray*}
for every $j\geq \widetilde j_{\varepsilon}$. Finally, since $\varepsilon>0$ is arbitrary% We can choose $\epsilon=\delta$, and sending this to $0$ 
 \ we conclude that

\begin{equation}\label{limsup}
\limsup_{j\to\infty}\int_{\mathbb{R}}|(\widetilde M_{\beta}f_{j})'(x)|^{q}dx\leq \int_{\mathbb{R}}|(\widetilde M_{\beta}f)'(x)|^{q}dx.
\end{equation}

\noindent
On the other hand, using Lemma \ref{convergencia puntual donde discolan} and Fatou's Lemma, we have that 

\begin{equation}\label{liminf}
\liminf_{j\to\infty}\int_{\mathbb{R}}|(\widetilde M_{\beta}f_{j})'(x)|^{q}dx\geq \int_{\mathbb{R}}|(\widetilde M_{\beta}f)'(x)|^{q}dx.
\end{equation}

\noindent
Combining \eqref{limsup}, \eqref{liminf} and Remark \ref{remark1}, we conclude the proof of Theorem \ref{Main Theorem}.

\end{proof}

%%%%%%%%%%%%%%%%%%%%%%%%%%%%%%%%%%%%%%%%%%%%%%%%%%%%%%%%%%%%5

%With the stronger hypothesis that $\{f_j\}_{j=1}^{\infty} \subset l^1(\Z)$ and $f_j \to f$ in $l^1(\Z)$, the conclusion that $\big({M}f_j\big)' \to \big({M}f\big)'$ in $\ell^1(\Z)$ holds, as shown in \cite[Theorem 1]{CH} (both in the centered and uncentered cases). In the discrete setting, note that the space $W^{1,1}(\Z)$ is merely $l^1(\Z)$ with an equivalent norm. Therefore, we adopt here the convention that $W^{1,1}(\Z)-$continuity is the same as $l^1(\Z)-$continuity just described. In this regard, Theorem \ref{Thm2} is the natural extension of \cite[Theorem 1]{CH} for the uncentered case, establishing the qualitatively stronger $BV-$continuity. This naturally leaves the open problem:

%\smallskip

%\noindent{\bf Question D.} \emph{Let $M$ be the discrete centered Hardy-Littlewood maximal operator. Is the map $M :BV(\mathbb{Z}) \to BV(\mathbb{Z})$ continuous?}

%%%%%%%%%%%%%%%%%%%%%%%%%%%%%%%%%%%%%%%%%%%%%%%%5

\section{Discrete setting -- Proof of Theorem \ref{Thm2}}
Let $f\in BV(\Z)$ and $x\leq y$ integers, the average of $f$ in an interval $[x,y]$ is given by
$$
A_{[x,y]}f:=\frac{1}{y-x}\sum_{k=x}^{y}|f(k)|.
$$
We say that $r$ is a good radius for $f$ at the point $n$ if 
$$
Mf(n)=\frac{1}{2r+1}\sum_{k=-r}^{r}|f(n+k)|=A_{r}f(n).
$$ 
We say that an interval $[x,y]$ (with $x,y\in\Z$) is a local maximum for $f$ if 
$$
f(x-1)<f(x)=f(z)=f(y)>f(y+1)
$$
for every $z\in[x,y]$. Analogously, we say that an interval $[x,y]$ (with $x,y\in \Z$) is a local minimum for $f$ if
$$
f(x-1)>f(x)=f(z)=f(y)<f(y+1)
$$
for every $z\in[x,y]$. Finally, we say that $m\in\Z$ is a global maximum (respectively minimum) for $f$ if
$$
f(m)\geq f(n)\ \ \ (\text{respectively}\ \leq)\ \ \text{for every} \ n\in\Z.
$$

%For fixed $f\in BV(\Z)$ 
%we will consider the sequence of local maxima and local minima for $Mf$
%\begin{equation}\label{local max min}
%\dots, [b_{i_-},b_{i_+}], [a_{i_-},a_{i_+}], [b_{(i+1)_-},b_{(i+1)_+}], [a_{(i+1)_-},a_{(i+1)_+}], \dots,
%\end{equation}
%where $[a_{i_-},a_{i_+}]$ denotes a local maximum of $Mf$  and $[b_{i_-},b_{i_+}]$ denotes a local minimum of $Mf$ for every $i\in\Z$, and $\dots< b_{i_-}\leq b_{i_+}<a_{i_-}\leq a_{i_+}<b_{(i+1)_-}\leq b_{(i+1)_+}<\dots$

{We start proving some useful lemmas.}
\begin{lemma}\label{uniform convergence disc}
Given a function $f\in BV(\mathbb{Z})$, and a sequence $\{f_{j}\}_{j\in\mathbb{N}}\subset BV(\mathbb{Z})$, such that $\|f-f_j\|_{BV(\Z)}\to0$ as $j\to\infty$, %and $c=\lim_{n\to\infty} f(n)=\lim_{n\to\infty}f_j(n)$ for every $j$, 
then
\begin{equation*}
\|Mf-Mf_j\|_{\ell^{\infty}(\mathbb{Z})}\to0 \ \text{as}\ \ j\to\infty.
\end{equation*}
\end{lemma}
\begin{proof}[Proof of Lemma \ref{uniform convergence disc}]
Given $m\in\Z$ we have that
\begin{eqnarray*}
f(m)-f_j(m)&=&(f(m)-f_j(m))-(f(n)-f_j(n))+(f(n)-f_j(n))\\
&\leq&Var(f-f_j)+(f(n)-f_{j}(n))
\end{eqnarray*}
for every $n$, thus
\begin{eqnarray*}
f(m)-f_j(m)&\leq& Var(f-f_j)+\lim_{n\to-\infty}(f(n)-f_j(n))\\
&=&\|f-f_j\|_{BV(\Z)}.
\end{eqnarray*}
Therefore, by the sublinearity of $M$,
$$
\|Mf-Mf_j\|_{\ell^{\infty}(\Z)}\leq\|f-f_j\|_{\ell^{\infty}(\Z)}\leq \|f-f_j\|_{BV(\Z)}\to0\ \ \text{as} \ \ j\to\infty.
$$
\end{proof}

%%%%%%%%%%%%%%%%%%%%%%%%%%%%%%%%%%%%%%%
\noindent As a consequence of Lemma \ref{uniform convergence disc}.
\begin{corollary}\label{point conv deriv} Under the hypotheses of Lemma \ref{uniform convergence disc}, we have
$$
\|(Mf)'-(Mf_j)'\|_{\ell^{\infty}(\Z)}\to0\ \ \text{as}\ \ j\to\infty.
$$
\end{corollary}
\noindent Therefore, by the Brezis--Lieb Lemma, we see that to obtain Theorem \ref{Main Theorem} it is enough to prove that
\begin{equation}\label{BL disc}
\lim_{j\to\infty}\|(Mf_j)'\|_{\ell^{1}(\Z)}=\|(Mf)'\|_{\ell^{1}(\Z)}.
\end{equation}

%%%%%%%%%%%%%%%%%%%%%%%%%%%%%%%%%%%%55
\begin{lemma}\label{radi}
Let $f:\Z\to\RR$ be a function in $ BV(\Z)$. 
\begin{itemize}
\item If $Mf(m)<\infty$ for some $m\in\Z$ then $Mf(n)<\infty$ for every $n\in\Z$.
\item If $n$ is not a global minimum for $Mf$, then there is a radius $r_{n}$ such that
$$
Mf(n)=A_{r_n}f(n)=\frac{1}{2r_{n}+1}\sum_{k=-r_{n}}^{r_{n}}|f(n+k)|.
$$ 
\end{itemize}
\end{lemma}
\noindent From now on for every $n\in\Z$ such that $n$ is not a global minimum for $Mf$ we will denote by $r_{n}$ a fixed good radius for $f$ at the point $n$.
\begin{proof}[Proof of Lemma \ref{radi}]
Both of the items follow as a consequence of the next inequality:
%(i) If there is $n \in \Z$ such that ${M}_{\beta}f(n) =\infty$, there exists a sequence $\{r_j, s_j\}$ in $\Z^+ \times \Z^+$, with $r_j + s_j \to \infty$ such that $A_{r_j,s_j}f(n) \to \infty$ as $j \to \infty$. 
%For each $m \in \Z$, letting  we have
\begin{equation}\label{Sec4_lem6_eq2}
A_{r}f(m) \geq A_{r}f(n) - \frac{2 C |m-n|}{2r+1}\,,
\end{equation}
valid for any points $n,m\in\Z$ and any radius $r\in\Z$, where $C = \|f\|_{\ell^{\infty}(\Z)}$ \cite[Lemma 6]{CaMa}.

%\smallskip

%\noindent(ii) If ${M}_{\beta}f(n)$ is not attained, there exists a sequence $\{r_j, s_j\}$ in $\Z^+ \times \Z^+$, with $r_j + s_j \to \infty$ such that $A_{r_j,s_j}f(n) \to {M}_{\beta}(n)$ as $j \to \infty$. The inequality \eqref{Sec4_lem6_eq1} plainly follows from \eqref{Sec4_lem6_eq2}.
\end{proof}

%%%%%%%%%%%%%%%%%%%%%%%%%%%%%%%%%%

%The next Lemma was proved by Temur \cite[Lemma 1]{Te} 
%\begin{lemma}\label{temur1}
%If 
%\end{lemma}

The next lemma will be crucial in the proof of Theorem \ref{Thm2}, and it follows from Kurka's ideas in \cite{Ku}. 
\begin{lemma}\label{existencia de s}
Given a function $f:\Z\to\RR$, if $[a_{-},a_{+}]$ is a local maximum of $Mf$ such that $Mf(a)\neq f(a)$ for every $a\in [a_{-},a_{+}]$ and $[b_{-},b_{+}]$ is a local minimum of $Mf$, such that $b_{+}<a_{-}$ and $Mf$ is monotone in $[b_{+},a_{-}]$, then there is $s\in[b_{+},a_{+}+r_{a_{+}}]$ such that
$$
|f(s)|\geq Mf(b_{+})+\frac{Mf(a_{+})-Mf(b_{+})}{2(a_{+}-b_{+})}(2r_{a_{+}}+1).
$$ 
\end{lemma}
\begin{proof}[Proof of Lemma \ref{existencia de s}]
First of all we see that
$
a_+-r_{a_+}<b,
$
since otherwise  taking a point $c\in(a_+,a_++r_{a_+})$ such that $Mf(n)<Mf(a_+)$ for every $n\in(a_++1,c]$, we would have that
$$
Mf(c)<Mf(a_+)=A_{r_{a_+}}f(a_+)\leq \max\{A_{[a_+-r_{a_+},c-(a_++r_{a_+}-c)]}f,Mf(c)\},
$$
which implies
$$
Mf(a_+)=A_{[a_+-r_{a_+},c-(a_++r_{a_+}-c)]}f\leq A_{[a_+-r_{a_+},c-(a_++r_{a_+}-c)]}Mf\leq Mf(a_+),
$$
thus $f(n)=Mf(n)=Mf(a_+)$ for every $n\in [a_+-r_{a_+},c-(a_++r_{a_+}-c)]$ which is a contradiction with the assumptions.

Then
\begin{eqnarray*}
&&(2r_{a_+}+1)Mf(a_+)-(2(b_+-(a_+-r_{a_+}))+1)Mf(b_+)\\
&&\leq (2r_{a_+}+1)Mf(a_+)-(2(b_+-(a_+-r_{a_+}))+1)A_{b_+-(a_+-r_{a_+})}|f|(b_+)\\
&&\leq \sum_{k=2b_+-(a_+-r_{a_+})}^{a_++r_{a_+}}|f(k)|,
\end{eqnarray*}
which implies
$$
(2r_{a_+}+1)(Mf(a_+)-Mf(b_+))+2(a_+-b_+)Mf(b_+)\leq \sum_{k=2b_+-(a_+-r_{a_+})}^{a_++r_{a_+}}|f(k)|.
$$
Thus
\begin{equation*}
\frac{(2r_{a_+}+1)}{2(a_+-b_+)}(Mf(a_+)-Mf(b_+))+Mf(b_+)\leq \frac{1}{2(a_+-b_+)}\sum_{k=a_++r_{a_+}-2(a_+-b_+)}^{a_++r_{a_+}}|f(k)|,
\end{equation*}
and therefore we can choose $s\in[a_++r_{a_+}-2(a_+-b_+),a_++r_{a_+}]$ such that
$$
|f(s)|\geq f(k)|\ \ \text{for every} \ \ k\in[a_++r_{a_+}-2(a_+-b_+),a_++r_{a_+}].
$$
The conclusion follows from this.
\end{proof}

%%%%%%%%%%%%%%%%%%%%%%%%%%%%%%%%%%%
%\begin{lemma}\label{two cases}
%a
%\end{lemma}
%\begin{proof}[Proof of Lemma \ref{two case}]
%a
%\end{proof}

%%%%%%%%%%%%%%%%%%%%%%%%%%%%%%%%%%%%%%%%%%%%%%%%%%%%
\begin{proof}[Proof of Theorem \ref{Thm2}]
Using Corollary \ref{point conv deriv} by Fatou's Lemma
we have
$$
\liminf_{j\to\infty}\|(Mf_j)'\|_{\ell^{1}(\Z)}\geq\|(Mf)'\|_{\ell^{1}(\Z)}.
$$
Thus, by \eqref{BL disc}, in order to obtain Theorem \ref{Thm2} it is sufficient to prove that
\begin{equation}\label{goal}
\limsup_{j\to\infty}\|(Mf_j)'\|_{\ell^{1}(\Z)}\leq\|(Mf)'\|_{\ell^{1}(\Z)}.
\end{equation}

Given $\delta>0$ there is $k=k(\delta)\in\Z$ such that
$$
Var_{[-k,k]^{c}}(f)<\delta\ \ \text{and}\ \ Var_{[-k,k]^{c}}(Mf)<\delta.
$$
Moreover, by the hypothesis and Corollary \ref{point conv deriv}  given $\epsilon>0$ there is $j_\epsilon$ such that
$$
Var(f-f_j)<\epsilon\ \ \text{and}\ \ \|Mf-Mf_j\|_{\ell^{\infty}(\Z)}<\epsilon
$$
for every $j\geq j_\epsilon$. From now on we assume that $j\geq j_\epsilon$, thus
\begin{equation}\label{compact}
Var_{[-k,k]}(Mf-Mf_j)\leq 2(2k+1)\epsilon. 
\end{equation}

\noindent Moreover
\begin{eqnarray*}
Var_{[k,\infty]}(Mf_j)&=&\|(Mf_{j})'\|_{l^{1}([k,\infty])}\\
&=&\|(Mf_{j})'\|_{l^{1}(\{n\in[k,\infty],k\leq n-r_n\})}\\
&&+\|(Mf_{j})'\|_{l^{1}(\{n\in[k,\infty],n-r_n<k\})}\\
&=&I+II.
\end{eqnarray*}

\noindent
To estimate $I$ we use the boundedness result proved by Temur in \cite{Te} (i.e $Var(Mg)\leq CVar(g)$ for any $g\in BV(\Z)$ where $C$ is a universal constant).
\begin{eqnarray*}
&&\|(Mf_{j})'\|_{l^{1}(\{n\in[k,\infty],k\leq n-r_n\})}\\
&&\leq \|(M(f_{j}\chi_{[k,\infty]}))'\|_{l^{1}(\Z)}\\
&&\leq C\|(f_{j}\chi_{[k,\infty]})'\|_{l^{1}(\Z)}\\
&&\leq C(\|(f_{}\chi_{[k,\infty]})'\|_{l^{1}(\Z)}+\epsilon)\\
&&\leq C(\delta+\epsilon).
\end{eqnarray*}

\noindent
Now we need to estimate $II$. %see that $Var_{{\{n\in[k,\infty]}, n-r_n\leq k\}}Mf_j$ is small. 

%\item 
{\remark{
If $Mf_j$ is non-increasing in ${\{n\in[k,\infty]}, n-r_n\leq k\}$ we can conclude the desired result. Since in that situation there is $a\in[k,\infty]$ such that
\begin{eqnarray*}
\|(Mf_j)'\|_{l^{1}(n\in[k,\infty], n-r_n\leq k)}&\leq& Mf_j(a)-\lim_{n\to\infty}Mf_j(n)\\
&\leq&Mf(a)-\lim_{n\to\infty}Mf(n)+2\epsilon\\
&\leq&Var_{[k,\infty]}Mf+2\epsilon\\
&\leq&\delta+2\epsilon.
\end{eqnarray*}}}
In general, there are two possibilities:\\
%\begin{itemize}
%\item 
{\it{Case 1:}} If $Mf_{j}(k)-\lim_{n\to\infty}Mf_j(n)\geq \frac{Var_{{\{[k,\infty]}%, n-r_n%\leq k
\}}Mf_j}{2}$. %for some $a\in[k,\infty]$ 
We can treat this case as in the previous situation. Since for similar argument we get
$$
Var_{[k,\infty]}Mf_j\leq2(Var_{[k,\infty]}Mf+2\epsilon)\leq 2(\delta+2\epsilon). 
$$ %%%%%%%%%%%5%%%%%%%%%%%%%%%%%%%%%%%%%%%%%%%%5%%%%%%
%\item 
{\it{Case 2:}} If $Mf_{j}(k)-\lim_{n\to\infty}Mf_j(n)< \frac{Var_{{\{[k,\infty]\}}}Mf_j}{2}$. This is the most complicated case. We will consider the sequence of local maxima and local minima for $Mf_j$ in $[k,\infty]$
\begin{equation}\label{local max min}
\dots, [b_{i_-},b_{i_+}], [a_{i_-},a_{i_+}], [b_{(i+1)_-},b_{(i+1)_+}], [a_{(i+1)_-},a_{(i+1)_+}], \dots,
\end{equation}
where $[a_{i_-},a_{i_+}]$ denotes a local maximum of $Mf_j$  and $[b_{i_-},b_{i_+}]$ denotes a local minimum of $Mf_j$ for every $i\in\Z$, and $\dots< b_{i_-}\leq b_{i_+}<a_{i_-}\leq a_{i_+}<b_{(i+1)_-}\leq b_{(i+1)_+}<\dots$

%We consider $\dots<b_i<a_{i}<\dots$ the sequence of local maxima and local minima for $Mf_j$ in $[k,\infty]$ 
%
Given $u\in (k,\infty)$ we can consider the terms in the list \eqref{local max min} lying in the interval $[k,u]$ (if $k$ and $u$ are not appearing in the list \ref{local max min}, for convenience  we include these terms in the list), we see that, if 
$$
S_{1,j}(k,u):=\sum_{\{i, [b_{i_+},a_{i_+}]\subset[k,u]\}}Mf_{j}(a_{i_+})-Mf_{j}(b_{i_+})
$$
and
$$
S_{2,j}(k,u):=\sum_{\{i, [a_{i_+},b_{(i+1)_+}]\subset[k,u]\}}Mf_{j}(a_{i_+})-Mf_{j}(b_{(i+1)_+}),
$$
\noindent using the fact that $Var Mf_{j}<\infty$ we have that
$$
S_{1,j}(k):=S_{1,j}(k,\infty)=\lim_{u\to\infty}S_{1,j}(k,u),
$$
and also
$$
S_{2,j}(k):=S_{2,j}(k,\infty)=\lim_{u\to\infty}S_{2,j}(k,u).
$$
\noindent Then
\begin{eqnarray*}
&&|S_{1,j}(k,u)-S_{2,j}(k,u)|\\
&&=|Mf_{j}(k)-Mf_{j}(u)|\\
&&\ \ \ \to |Mf_{j}(k)-\lim_{u\to\infty}Mf_{j}(u)| \ \text{as}\ \ u\to\infty.
\end{eqnarray*}
Thus, using the hypotheses,
$$
|S_{1,j}(k)-S_{2,j}(k)|=|Mf_{j}(k)-\lim_{u\to\infty}Mf_{j}(u)|<\frac{Var_{[k,\infty]}Mf_j}{2}.
$$
Moreover,
$$
Var_{[k,\infty]}Mf_j=S_{1,j}(k)+S_{2,j}(k).
$$
Therefore
$$
S_{1,j}(k)\geq \frac{Var_{[k,\infty]}Mf_j}{4} \ \text{and} \ \ S_{2,j}(k)\geq \frac{Var_{[k,\infty]}Mf_j}{4}.
$$
Using the first one of the two previous inequalities we obtain
\begin{eqnarray}\label{separacion}
Var_{[k,\infty]}Mf_j&\leq& 4\sum_{\{i,[b_{i_+},a_{i_+}]\subset[k,\infty]\}}Mf_j(a_{i_+})-Mf_j(b_{i_+})\\
&\leq&4Var_{[k,\infty]}f_j+I\\
&&+4\sum_{i\in D_j([k,\infty])}Mf_j(a_{i_+})-Mf_j(b_{i_+}).\nonumber
\end{eqnarray}
\noindent Here $D_j([u,v])=\{i,[b_{i_+},a_{i_+}]\subset[u,v],\  Mf_j(a_{i_+})\neq f_j(a_i)\ \ \text{for every}\ \ a_{i}\in [a_{i_-},a_{i_+}] \ \ \text{and}\ a_{i_{+}}-r_{a_{i_{+}}}\leq k\}$, for every $u,v\in(k,\infty]$. We consider $a_{{j_1}_+}$ such that ${j_{1}}\in D_{j}([k,\infty])$ and
\begin{equation}\label{step1}
\frac{Mf_j(a_{{j_1}_+})-Mf_j(b_{{j_1}_+})}{a_{{j_1}_+}-b_{{j_1}_+}}\geq\frac{Mf_j(a_{i_+})-Mf_j(b_{i_+})}{a_{i_+}-b_{i_+}} \ \text{for every}\ \ i\in D_{j}([k,\infty]).
\end{equation}
It does exist because 
$$
\frac{Mf_j(a_{i_+})-Mf_j(b_{i_+})}{a_{i_+}-b_{i_+}}\leq \|(Mf_j)'\|_{\ell^{\infty}[b_{i_+},a_{i_+}]}\to0 \ \ \text{as}\ \ i\to\infty.
$$

Therefore using Lemma \ref{existencia de s} %and the fact that $a_{i_+}-r_{a_{i_+}}\leq $ 
we have
\begin{eqnarray}\label{case no monotone}
&&\sum_{i\in D_j([k,a_{{j_1}_+}+r_{a_{{j_1}_+}}])} Mf_j(a_{i_+})-Mf_j(b_{i_+})\nonumber\\
&&\ \ \ \ \ \ \ \ \ =\sum_{{i\in D_j([k,a_{{j_1}_+}+r_{a_{{j_1}_+}}])}} \frac{Mf_j(a_{i_+})-Mf_j(b_{i_+})}{a_{i_+}-b_{i_+}}(a_{i_+}-b_{i_+})\nonumber\\
&&\ \ \ \ \ \ \ \ \ \leq\sum_{{i\in D_j([k,a_{{j_1}_+}+r_{a_{{j_1}_+}}])}} \frac{Mf_j(a_{{j_1}_+})-Mf_j(b_{{j_1}_+})}{a_{{j_1}_+}-b_{{j_1}_+}}(a_{i_+}-b_{i_+})\nonumber\\
&&\ \ \ \ \ \ \ \ \ =\frac{Mf_j(a_{{j_1}_+})-Mf_j(b_{{j_1}_+})}{a_{{j_1}_+}-b_{{j_1}_+}}\sum_{{i\in D_j([k,a_{{j_1}_+}+r_{a_{{j_1}_+}}])}} (a_{i_+}-b_{i_+})\nonumber\\
&&\ \ \ \ \ \ \ \ \ \leq\frac{Mf_j(a_{{j_1}_+})-Mf_j(b_{{j_1}_+})}{a_{{j_1}_+}-b_{{j_1}_+}}(2r_{a_{{j_1}_+}}+1)\nonumber\\
&&\ \ \ \ \ \ \ \ \ \leq 2(f_j(s_{j_1})-f_j(b_{{j_{1}}_+}))\nonumber\\
&&\ \ \ \ \ \ \ \ \ \leq 2Var_{[k,a_{{j_1}_+}+r_{a_{{j_1}_+}}]}f_j.
\end{eqnarray}
For some $s_{j_1}\in [b_{{j_1}_+}+(b_{{j_1}_+}-(a_{{j_1}_+}-r_{a_{{j_1}_+}})),a_{{j_{1}}_+}+r_{a_{{j_1}_+}}]$. Then we look at the local maximum $[a_{{t_1}_-},a_{{t_{1}}_+}]$ of $Mf_j$ such that $Mf_{j}$ is monotone between $a_{{t_1}_+}$ and $a_{{j_1}_+}+r_{a_{{j_1}_+}}$. If $Mf_j$ is a non-increasing function in $[a_{{t_1}_+},\infty]$ then
\begin{eqnarray}\label{case monotone}
Var_{[k,\infty]}(Mf_j)&\leq& 2Var_{[k,a_{{j_1}_+}+r_{a_{{j_1}_+}}]}(f_j)+Var_{[a_{{t_1}_+},\infty]}(Mf)+2\epsilon\nonumber\\
&\leq& 3\delta+4\epsilon.
\end{eqnarray}
Otherwise we consider a local maximum $[a_{{j_2}_-},a_{{j_2}_+}]\in [a_{{t_1}_+}  ,\infty]$ such that
\begin{eqnarray*}
\frac{Mf_j(a_{{j_2}_+})-Mf_j(b_{{j_2}_+})}{a_{{j_2}_+}-b_{{j_2}_+}}\geq\frac{Mf_j(a_{i_+})-Mf_j(b_{i_+})}{a_{i_+}-b_{i_+}} 
\end{eqnarray*}
{for every}\ \ $[a_{i_-},a_{i_+}]\in[a_{{t_1}_+},\infty]$,
similarly to \eqref{step1}. Then, following the same analysis presented  previously, we have that
\begin{eqnarray*}
&&\sum_{i\in D_j([a_{{t_1}_+},a_{{j_2}_+}+r_{a_{{j_2}_+}}])} Mf_j(a_{i_+})-Mf_j(b_{i_+})\\
&&\ \ \ \ \ \ \ \ \ \leq 2(f_j(s_{j_2})-f_j(b_{{j_{2}}_+}))\\
&&\ \ \ \ \ \ \ \ \ \leq 2Var_{[a_{{t_1}_+},a_{{j_2}_+}+r_{a_{{j_2}_+}}]}f_j,
\end{eqnarray*}
for some $s_{j_2}\in[a_{{t_1}_+},a_{{j_2}_+}+r_{a_{{j_2}_+}}]$.
%\end{itemize}
We can proceed inductively. %this way, 
Having defined $([a_{{j_1}_-},a_{{j_1}_+}], s_{j_1}, [a_{{t_1}_-},a_{{t_1}_+}])$, $([a_{{j_2}_-},a_{{j_2}_+}], s_{j_2}, [a_{{t_2}_-},a_{{t_2}_+}])\dots$ $([a_{{j_i}_-},a_{{j_i}_+}], s_{j_i}, [a_{{t_i}_-},a_{{t_i}_+}])$ we have that in case $Mf_j$ is a non-increasing function in $[a_{{t_i}_+},\infty]$ we use the same reasoning as in \eqref{case monotone} to conclude that
\begin{eqnarray*}
Var_{[k,\infty]}Mf_j&\leq& 2Var_{[k,a_{{j_1}_+}+r_{a_{{j_1}_+}}]}f_j+2\sum_{l=2}^{i}Var_{[a_{{t_{l-1}}_+},a_{{j_l}_+}+r_{a_{{j_l}_+}}]}f_j\\&&+Var_{[a_{t_i},\infty]}Mf_j\nonumber\\
&\leq&4Var_{[k,a_{{j_i}_+}+r_{a_{{j_i}_+}}]}f_j+Var_{[a_{{t_i}_+},\infty]}Mf+2\epsilon\nonumber\\
&\leq& 5\delta+6\epsilon,
\end{eqnarray*}
we have used the fact that by construction there is a probable overlap just between consecutive intervals. Otherwise, $Mf_j$ is not a non-increasing function in $[a_{{t_i}_+},\infty)$ thus
we consider a local maximum $[a_{{j_{i+1}}_-},a_{{j_{i+1}}_+}]\in [a_{{t_i}_+}  ,\infty]$ such that
\begin{eqnarray*}
\frac{Mf_j(a_{{j_{i+1}}_+})-Mf_j(b_{{j_{i+1}}_+})}{a_{{j_{i+1}}_+}-b_{{j_{i+1}}_+}}\geq\frac{Mf_j(a_{l_+})-Mf_j(b_{l_+})}{a_{l_+}-b_{l_+}} 
\end{eqnarray*}
{for every}\ \ $[a_{l_-},a_{l_+}]\in[a_{{t_i}_+},\infty]$
then we can look at the local maximum $[a_{{t_{i+1}}_-},a_{{t_{i+1}}_+}]$ of $Mf_j$ such that $Mf_j$ is monotone between $a_{{t_{i+1}}_+}$ and $a_{{j_{i+1}}_+}+r_{a_{{j_{i+1}}_+}}$ and we can get an estimative like \eqref{case no monotone} and continue with the process.

 After all we conclude that
\begin{eqnarray}\label{Dj}
&&\sum_{i\in D_j([k,\infty])}Mf_j(a_i)-Mf_j(b_i)\\
&&\ \ \ \leq 4Var_{[k,\infty]}(f_j)+Var_{[k,\infty]}(Mf)+2\epsilon\nonumber\\
&&\ \ \ \leq 5\delta+6\epsilon. \nonumber
\end{eqnarray}

\noindent
Therefore, combining \eqref{separacion} with \eqref{Dj}, we obtain
\begin{equation}\label{conclusion derecha}
Var_{[k,\infty]} (Mf_j)%\leq 16Var_{[k,\infty]}f_j+4Var_{[k,\infty]}Mf+8\epsilon
\leq 20\delta+24\epsilon.
\end{equation}
Analogously,
\begin{equation}\label{conclusion izquierda}
Var_{[-\infty,-k]} (Mf_j)\leq 20\delta+24\epsilon.%12Var_{[-\infty,-k]}f_j+4Var_{[-\infty,-k]}Mf+8\epsilon.
\end{equation}
Finally, combining \eqref{compact},\eqref{conclusion derecha} and \eqref{conclusion izquierda}  we have that
$$
Var (Mf_j)\leq Var_{[-k,k]}(Mf)+(4k+2)\epsilon+40\delta+48\epsilon.
$$
Sending $\epsilon\to0$ it implies that
$$
\limsup_{j\to\infty}Var (Mf_j)\leq Var_{}(Mf)+40\delta.
$$
Now, sending $\delta\to0$ we obtain \eqref{goal}, and conclude the desired result.
\end{proof}

\section{Acknowledgments}
\noindent The author thanks Emanuel Carneiro, Juha Kinnunen and Hannes Luiro for inspiring discussions. %The author is thankful to  for helpful discussions and guidance during the preparation of this manuscript, 
The author is also thankful to Academy of Finland, Aalto University and The Abdus Salam International Centre for Theoretical Physics (ICTP) for their support. %The authors also thank E. Carneiro for suggesting to think about this problem.

\end{document}